\documentclass[a4paper,twoside,DIV=12,headings=small,%
headsepline=true,toc=bib
]{scrartcl}
\pdfoutput=1
\ifpdfoutput{\usepackage{cmap}}{}
\usepackage[utf8]{inputenc}
\usepackage[T1,T2A]{fontenc}
\usepackage[inner=2cm,outer=2cm,top=2cm,bottom=2cm,footskip=.5cm,headsep=.3cm]{geometry}

\def\thispapertitle {Hurwitz matrices of doubly infinite series}

\usepackage{mathtools}
\usepackage{paralist}
\usepackage{colonequals}

\usepackage[usenames,svgnames]{xcolor}
\usepackage{float}
\usepackage
{graphicx}

\usepackage{extarrows}
\usepackage{mathdots}
\usepackage[selected=true,shrink=10,stretch=40,verbose=true,tracking=alltext,letterspace=0
]{microtype}
\SetTracking[
    spacing = {170*,-25*,}
]{ font = {*/PTSerif-TLF/*/it/*} }{ 10 }

\usepackage{setspace}
\usepackage[permil]{overpic}

\ifpdfoutput{
\definecolor{mybrown}{HTML}{D02000}
\usepackage[colorlinks=false,linkcolor=mybrown, pdfborderstyle={/S/D/D[2 1]/W .5},
citecolor=DarkGreen,linktocpage=true,unicode=true,pdftex,
        pdfversion=1.5,
	pdftitle={\thispapertitle},
	pdfauthor={Alexander Dyachenko},
	pdfsubject={MSC(2010) 30C15, 30B10, 40A05},
	pdfkeywords={Total positivity, Polya frequency sequence, Total nonnegativity,
            Hurwitz matrix, Generalized Hurwitz matrix, Doubly infinite series},
	pdfpagemode=UseOutlines,pdflang=en-GB,
        ]{hyperref}
        \hypersetup{
            pdftitle={\thispapertitle}
        }
}
{\usepackage[colorlinks=true,citecolor=DarkGreen,linktoc=page,unicode=true
]{hyperref}
}
\usepackage[russian,british]{babel}
\usepackage{amsmath,amsthm}
\usepackage{amssymb,amsfonts}
\let\vvarkappa\varkappa
\usepackage
{fourier}
\usepackage{bm}
\usepackage{caption}
\usepackage{subcaption}
\usepackage{wrapfig}
\let\varkappa\vvarkappa

\RequirePackage[scaled=0.94]{PTSansCaption}  
\RequirePackage[scaled=0.94]{PTSerifCaption} 
\RequirePackage[scaled=0.94]{PTSans}
\RequirePackage[scaled=0.94]{PTSerif}

\ifpdfoutput
{\DeclareGraphicsExtensions{.pdf}}
{\DeclareGraphicsExtensions{.eps}}

\newcommand{\navy}{}

\newcommand{\gree}{}

\newcommand{\Arg}{\operatorname{Arg}}
\newcommand{\Ln}{\operatorname{Ln}}
\newcommand{\sign}{\operatorname{sign}}

\renewcommand{\le}{\leqslant}
\renewcommand{\ge}{\geqslant}

\newcommand{\ww}{\quad\text{where}\quad}
\newcommand{\an}{\quad\text{and}\quad}

\renewcommand{\Re}{\operatorname{Re}}
\renewcommand{\Im}{\operatorname{Im}}

\newtheorem{theorem}{\usekomafont{subparagraph}Theorem}

\newtheorem{lemma}[theorem]{\usekomafont{subparagraph}Lemma}

\newtheorem{corollary}[theorem]{\usekomafont{subparagraph}Corollary}

\theoremstyle{definition}
\newtheorem{fact}{\usekomafont{subparagraph}Fact}

\newtheorem{remark}[theorem]{\usekomafont{subparagraph}Remark}
\newtheorem*{remark*}{\usekomafont{subparagraph}Remark}

\newtheorem*{definition}{\usekomafont{subparagraph}Definition}

\numberwithin{paragraph}{section}

\makeatletter
\def\blfootnote{\gdef\@thefnmark{}\@footnotetext}
\makeatother

\title{\thispapertitle%
    \thanks{This work was supported by the Einstein Foundation Berlin and the European Research
        Council under the European Union's Seventh Framework Programme (FP7/2007--2013)/ERC
        grant agreement no.~259173.} }

\author{\gree\fontfamily{PTSansCaption-TLF}\selectfont\normalsize\vspace{1em}Alexander
    Dyachenko%
    \footnote{\emph{TU-Berlin, Institut f\"ur Mathematik, Sekr.~MA 4-2, Straße des 17. Juni 136,
        10623 Berlin, Germany.\newline        
        \href{mailto:diachenko@sfedu.ru}{diachenko@sfedu.ru},
        \href{mailto:dyachenk@math.tu-berlin.de}{dyachenk@math.tu-berlin.de} }}
}
\date{\gree\fontfamily{PTSansCaption-TLF}\selectfont\vspace{-.4em}{\small\today\vskip -2.5em}}

\usepackage{scrpage2}
\clearscrheadfoot
\ohead{\pagemark}
\cehead{\quad A.\,Dyachenko --- \thispapertitle}
\cohead{\leftmark}
\setheadsepline{.4pt}
\deftripstyle{my_titlepagestyle}{}{}{}{}{\pagemark}{}

\addtokomafont{title}{\navy\fontfamily{PTSans-TLF}\selectfont}
\addtokomafont{section}{\navy\fontfamily{PTSansCaption-TLF}\selectfont\mdseries\large}
\addtokomafont{subsection}{\navy\fontfamily{PTSansCaption-TLF}\selectfont\mdseries}
\addtokomafont{paragraph}{\navy\fontfamily{PTSansCaption-TLF}\selectfont\mdseries}
\addtokomafont{subparagraph}{\navy\fontfamily{PTSans-TLF}\selectfont}
\addtokomafont{pageheadfoot}{\gree\fontfamily{PTSansCaption-TLF}\selectfont}
\addtokomafont{pagenumber}{\gree\fontfamily{PTSansCaption-TLF}\selectfont}
\addtokomafont{headsepline}{\gree}
\addtokomafont{footsepline}{\gree}

\setkomafont{sectionentry}{\vspace{-.7em}\fontsize{8pt}{1em}\fontfamily{PTSansCaption-TLF}\selectfont}

\makeatletter
\renewcommand{\tocbasic@@before@hook}{\vspace{-1.5em}\parskip0pt\itemsep0pt\setstretch{1}}
\makeatother

\newcommand{\Cp}{\ensuremath{\mathbb{C}_+}}

\begin{document}
\automark{section}
\maketitle

\setstretch{1.2}

\begin{abstract}
    This paper aims at extending the criterion that the quasi-stability of a polynomial is
    equivalent to the total nonnegativity of its Hurwitz matrix. We give a complete description
    of functions generating doubly infinite series with totally nonnegative Hurwitz and
    Hurwitz-type matrices (in a Hurwitz-type matrix odd and even rows come from two distinct
    power series). The corresponding result for singly infinite series is known: it is based on
    a certain factorization of Hurwitz-type matrices, which is absent in the doubly infinite
    case. A necessary condition for total nonnegativity of generalized Hurwitz matrices follows
    as an application.
    \\[10pt]
        \\[5pt]
\end{abstract}
\section{Introduction}
\begin{definition}
    A doubly (\emph{i.e.} two-way) infinite sequence $\big(f_n\big)_{n=-\infty}^\infty$ is
    called \emph{totally positive} if all minors of the (four-way infinite) Toeplitz matrix
    \begin{equation*}
        \begin{pmatrix}
            \ddots & \vdots & \vdots & \vdots & \vdots & \vdots & \iddots\\
            \hdots & f_0 & f_1 & f_2 & f_3 & f_4 & \hdots\\
            \hdots & f_{-1} & f_0 & f_1 & f_2 & f_3 & \hdots\\
            \hdots & f_{-2} & f_{-1} & f_0 & f_1 & f_2 & \hdots\\
            \hdots & f_{-3} & f_{-2} & f_{-1} & f_0 & f_1 & \hdots\\
            \hdots & f_{-4} & f_{-3} & f_{-2} & f_{-1} & f_0 & \hdots\\
            \iddots & \vdots & \vdots & \vdots & \vdots & \vdots & \ddots
        \end{pmatrix}
        \equalscolon T(f),
        \ww
        f(z)\colonequals\sum_{\mathclap{n=-\infty}}^\infty \ f_nz^n
    \end{equation*}
    are nonnegative (\emph{i.e.} the matrix is \emph{totally nonnegative}).
\end{definition}
Note that the indexation of four-way infinite matrices affects the multiplication. Here we adopt
the following convention: the uppermost row and the leftmost column, which appear in
representations of such matrices, have the index~$1$ unless another is stated explicitly.

The total nonnegativity of the corresponding Toeplitz matrices is a characteristic property of
power series converging to functions of a very specific form:
\begin{theorem}[{Edrei~\cite{Edrei}%
        \footnote{An earlier publication~\cite{AESW} studies the singly infinite case. Under
            additional conditions, for example~$\beta_{-\mu}^{-1}<1<\beta_\mu$ for each~$\mu>0$,
            the representation~\eqref{eq:funct_gen_dtps} is unique in the annulus of
            convergence, see \emph{e.g.}~\cite[Theorem~4]{Ganzburg}.}}]
    \label{th:E-AESW}
    Let a non-trivial sequence~$\big(f_n\big)_{n=-\infty}^\infty$ be totally positive. Then,
    unless~$f_n=f_0^{1-n}f_1^{n}$ for every~$n\in\mathbb{Z}$, the series~$f(z)$ converges in
    some annulus to a function with the following 
    representation
    \begin{equation}\label{eq:funct_gen_dtps}
        Cz^je^{Az+\frac{A_0}z}\cdot
        \frac{\prod_{\mu>0} \left(1+\frac{z}{\beta_\mu}\right)}
        {\prod_{\nu>0} \left(1-\frac{z}{\delta_\nu}\right)}\cdot
        \frac{\prod_{\mu<0} \left(1+\frac{z^{-1}}{\beta_\mu}\right)}
        {\prod_{\nu>0} \left(1-\frac{z^{-1}}{\delta_\nu}\right)},
    \end{equation}
    where the products converge absolutely, \(j\) is integer and the coefficients satisfy
    \(A,A_0\ge 0\), ~\(C,\beta_\mu,\delta_\nu>0\) for all~\(\mu,\nu\). The converse is also
    true: every function of this form generates (\emph{i.e.} its Laurent coefficients give) a
    doubly infinite totally positive sequence.
\end{theorem}
Recent publications~\cite{HoltzTyaglov,Dyachenko14} have shown that a relevant criterion holds
for the so-called Hurwitz-type matrices, which are built from two Toeplitz matrices and have
applications to questions of stability.
\begin{definition}
    The \emph{Hurwitz-type matrix} is a matrix of the form
    \begin{equation}\label{eq:Hpq_def}
        H(p,q)={\begin{pmatrix}
                \ddots & \vdots & \vdots & \vdots & \vdots & \vdots & \vdots & \iddots\\
                \hdots & a_0& a_1& a_2 & a_3 & a_4 & a_5 & \hdots\\
                \hdots & b_0& b_1& b_2 & b_3 & b_4 & b_5 & \hdots\\
                \hdots & a_{-1} & a_0 & a_1 & a_2 & a_3 & a_4 & \hdots\\
                \hdots & b_{-1} & b_0 & b_1 & b_2 & b_3 & b_4 & \hdots\\
                \hdots & a_{-2} & a_{-1} & a_0 & a_1 & a_2 & a_3 & \hdots\\
                \iddots & \vdots & \vdots & \vdots & \vdots & \vdots & \vdots & \ddots
            \end{pmatrix}},
    \end{equation}
    where~$p(z)=\sum_{k=-\infty}^\infty\,a_kz^k$ and~$q(z)=\sum_{k=-\infty}^\infty\,b_kz^k$ are
    formal power series.
\end{definition}
\begin{definition}
    The \emph{Hurwitz matrix} corresponding to a power
    series~$f(z)=\sum_{k=-\infty}^\infty\,f_kz^k$ is the Hurwitz-type matrix~$H(p,q)$ in which
    the series~$p(z)$ and~$q(z)$ are defined by~$f(z)=q(z^2)+zp(z^2)$.
\end{definition}
The main goal of the present study is to determine conditions on the power series~$p(z)$
and~$q(z)$ necessary and sufficient for total nonnegativity of the matrix~$H(p,q)$: like in the
case of singly infinite series, one of the conditions is that the ratio~$\frac{q(z)}{p(z)}$ maps
the upper half-plane~$\Cp\coloneqq\{z\in\mathbb C:\Im z>0\}$ into itself. To give a more precise
statement, let us introduce the following class of functions:
\begin{definition}
    A function~$F(z)$ is called an $\mathcal{S}$-function if it is holomorphic and
    satisfies~$\Im z\cdot\Im F(z)\ge 0$ for all~$z\not\le0$ and if additionally~$F(z)\ge0$
    wherever~$z>0$.
\end{definition}
The straightforward corollary of the definition is that~$F(\overline z)=\overline{F(z)}$ for
each~$\mathcal S$-function~$F(z)$ wherever it is regular. We need a subclass of
$\mathcal{S}$-functions introduced in the following lemma.
\begin{lemma}\label{lemma:prop_S1}
    Let~$p(z)$ and~$q(z)$ be two functions of the form~\eqref{eq:funct_gen_dtps}; then their
    ratio~$F(z)=\frac{q(z)}{p(z)}$ is an $\mathcal S$-function if and only if there exists a
    function~$g(z)$ of the form~\eqref{eq:funct_gen_dtps}, such that
    \[
    \begin{gathered}
        \frac{p(z)}{g(z)}= a_0{\prod_{\nu>0} \left(1+\dfrac{z}{\alpha_\nu}\right)}
        {\prod_{\nu<0}\left(1+\dfrac{z^{-1}}{\alpha_\nu}\right)}, \quad \frac{q(z)}{g(z)}=
        b_0{\prod_{\mu>0} \left(1+\dfrac{z}{\beta_\mu}\right)}
        {\prod_{\mu<0}\left(1+\dfrac{z^{-1}}{\beta_\mu}\right)} \an
        \\
        0<\cdots<\alpha_{-2}^{-1}<\beta_{-1}^{-1}<\alpha_{-1}^{-1}
        <\beta_{1}<\alpha_{1}<\beta_{2}<\alpha_2<\cdots;
    \end{gathered}
    \]
    if the sequence of~$\mu$ terminates on the left at~$\mu_0$, then~$\beta_{\mu_0}$ can be
    positive or zero%
    \footnote{When~$\beta_{\mu_0}=0$, the corresponding
        factor~$\Big(1+\frac{z}{\beta_{\mu_0}}\Big)$ needs to be replaced by the factor~$z$.}
    and the sequence of~$\nu$ also terminates on the left at~$\mu_0$.
\end{lemma}
We prove this lemma in the end of Section~\ref{sec:s-funct}; it is an analogue of a theorem due
to Kre\u{\i}n, see~\cite[p.~308]{Levin}. In other words, under the conditions of
Lemma~\ref{lemma:prop_S1} the function~$F(z)$ can be expressed as
in~\eqref{eq:order_alphas_betas} or~\eqref{eq:order_alphas_betas_mer} below. The chain
inequality means that zeros of~$\frac{p(z)}{g(z)}$ and~$\frac{q(z)}{g(z)}$ are
\emph{interlacing}, that is all zeros of each of the functions are real and separated by zeros
of another. Lemma~\ref{lemma:prop_S1} provides an alternative reformulation of the
item~\eqref{item:m1} in our main result:
\begin{theorem}\label{th:main1}
    If~$a_0\ne 0$, then the following conditions are equivalent:
    \begin{enumerate}[\upshape(a)]
    \item \label{item:m1}%
        The series~$p(z)=\sum_{k=-\infty}^\infty\,a_kz^k$
        and~$q(z)=\sum_{k=-\infty}^\infty\,b_kz^k$ converge in some common annulus to functions
        of the form~\eqref{eq:funct_gen_dtps} and their ratio~$F(z)=\frac{q(z)}{p(z)}$ is an
        $\mathcal S$-function.
    \item \label{item:m2}%
        The matrices~$T(Ap+Bq)$ and~$T(Aq+B\widetilde p)$,
        where~$\widetilde p(z)\coloneqq zp(z)$, are totally nonnegative for every choice of the
        numbers~$A,B\ge 0$, and~$T(p)$ has a nonzero minor of order~$2$.
    \item \label{item:m3}%
        The matrix~$H(p,q)$ is totally nonnegative and has a nonzero minor of order~$2$.
    \end{enumerate}
\end{theorem}
\begin{remark}\label{rem:degenerating_series}
    Let~$a_0\ne 0$. Then the totally nonnegative matrix~$H(p,q)$ has only zero minors of
    order~$2$ if and only if~$a_k= a_0^{1-k}a_1^{k}\ne 0$ and~$b_k=\frac{b_0}{a_0}a_k$ for
    all~$k$, as is stated in Corollary~\ref{cor:total-nonn-interl}. This case is excluded from
    Theorem~\ref{th:main1} as corresponding to the divergence of the power series~$p(z)$, see
    Theorem~\ref{th:E-AESW}.
\end{remark}
Both earlier works~\cite{HoltzTyaglov,Dyachenko14} exploit a relation to the matching moment
problem trough the Hurwitz transform (see e.g.~\cite[p.~44]{ChebMei}
or~\cite[p.~427]{HoltzTyaglov}). In turn, doubly infinite series do not allow conducting the
same procedure due to the lack of the matching moment problem. Accordingly, the corresponding
Hurwitz-type matrices have no induced factorizations. To get around this difficulty, we first
obtain the implication~\eqref{item:m3}$\implies$\eqref{item:m2} of Theorem~\ref{th:main1}. Then
Theorem~\ref{th:E-AESW} allows us to reduce the problem to studying ratios of functions of the
form~\eqref{eq:funct_gen_dtps}. The inclusion of Item~\eqref{item:m2} in Theorem~\ref{th:main1}
yields a generalization of a fact known for polynomials, see e.g.~\cite[Lemma~3.4]{Wa}. In the
related publication~\cite{Dyachenko16}, we aim at deriving properties directly from estimates of
minors of Hurwitz-type matrices.

By definition, a polynomial is \emph{quasi-stable} if it has no zeros in the right half of the
complex plane. It is known~\cite{Asner,Kemperman}, that Hurwitz matrices of quasi-stable
polynomials with positive leading coefficients are totally nonnegative. The relevant criterion
of total nonnegativity of Hurwitz matrices follows from Theorem~\ref{th:main1}; it is an
extension of the recent result~\cite[Theorem~1.1]{Dyachenko14} to doubly infinite series:
\begin{theorem}\label{th:main2}
    A non-trivial two-way series~$f(z)=\sum_{k=-\infty}^{\infty}f_kz^k=q(z^2)+zp(z^2)$ converges
    to a function of the form
    \begin{equation}\label{eq:f_weierstrass}
        g(z^2)\cdot z^r e^{Bz+\frac{B_0}z}
        \prod_{\lambda\ne0} \left(1+\frac{z^{\sign\lambda}}{\xi_\lambda}\right)
        \cdot
        \prod_{\nu\ne0} \left(1+\frac{z^{\sign\nu}}{\gamma_\nu}\right)
        \left(1+\frac{z^{\sign\nu}}{\overline\gamma_\nu}\right),
    \end{equation}
    where~$\xi_\lambda,\Im\gamma_\nu,\Re\gamma_\nu>0$ and~$B, B_0\ge0$ for
    all~$\lambda,\nu\ne0$, the function~$g(z)$ can be represented as
    in~\eqref{eq:funct_gen_dtps} and~$r$ is an integer, if and only if the corresponding Hurwitz
    matrix~$H(p,q)$ is totally nonnegative and has a nonzero minor of order at least two.%
    \footnote{When all minors of the matrix~$H(p,q)$ of order two turn to zero, the
        series~$f(z)$ must be trivial (\emph{i.e.} all its coefficients are zero) or divergent,
        see Remark~\ref{rem:degenerating_series}.}
\end{theorem}
Note that the expression~\eqref{eq:f_weierstrass} can be rewritten
as~$q_1(z)\cdot q_2\left(\frac 1z\right)$, where both~$q_1(z)$ and~$q_2(z)$ can be represented
in the form
\[
    Cz^re^{Az^2+Bz}\cdot
    \prod_{\lambda>0} \left(1+\frac{z}{\xi_\lambda}\right)
    \cdot
    \prod_{\nu>0} \left(1+\frac{z}{\gamma_\nu}\right)
    \left(1+\frac{z}{\overline\gamma_\nu}\right)
    \cdot
    \prod_{\mu>0} \frac{1}{1-\frac{z^2}{\delta_\mu^2}}
\]
\setstretch{1.15}%
with the same conditions on the coefficients, except that~$\Re\gamma_\nu\ge 0$ and
additionally~$A\ge 0$ and~$C, c_\mu>0$ for all~$\mu>0$. Theorem~1.1 of~\cite{Dyachenko14} says
that the Hurwitz matrices generated by the involved functions~$q_1(z)$ and~$q_2(z)$ are totally
nonnegative; the annulus of convergence of~$f(z)$ is the domain, where both power series
for~$q_1(z)$ and~$q_2\left(\frac 1z\right)$ converge. In other words, the relations between
limits of singly and doubly infinite series corresponding to totally nonnegative Hurwitz and
Toeplitz matrices are akin.

Another outcome of Theorem~\ref{th:main1} is an extension
of~\cite[Theorem~4]{HoltzKushelKhrushchev16} on total nonnegativity of the generalized Hurwitz
matrices, \emph{i.e.} the matrices defined by
\[
\big(f_{jM-i+1}\big)_{i,j=-\infty}^{\infty}
=
\begin{pmatrix}
    \ddots& \vdots&\vdots& \vdots & \vdots & \iddots\\
    \hdots & f_M & f_{2M} & f_{3M} & f_{4M} & \hdots\\
    \hdots & f_{M-1} & f_{2M-1} & f_{3M-1} & f_{4M-1} & \hdots\\
    \vdots & \vdots& \vdots & \vdots  & \vdots & \vdots\\
    \hdots &f_0 & f_{M} & f_{2M} & f_{3M} & \hdots\\
    \hdots &f_{-1} & f_{M-1} & f_{2M-1} & f_{3M-1} & \hdots\\
    \iddots &\vdots& \vdots & \vdots & \vdots & \ddots
\end{pmatrix}
,
\]
where~$M=1,2,\dots$ and~$f_k$ is the~$k$th coefficient of a power series. (The extension seems
to be new even for the singly infinite case.) Let~$\arg z$ denote the principal value of the
argument of a complex number~$z \ne0$. Given a doubly infinite
series~$f(z)=\sum_{k=-\infty}^{\infty}f_kz^k=\sum_{n=0}^{M-1}z^np_n(z^M)$ assume that the
corresponding generalized Hurwitz matrix is totally nonnegative. Then all the Hurwitz-type
matrices~$H(p_m,p_n)$, ~$n<m$, are totally nonnegative as submatrices
of~$(f_{jM-i})_{i,j=-\infty}^{\infty}$; this weaker property ensures that the function
represented by~$f(z)$ does not vanish in a certain sector of the complex plane:
\begin{theorem}\label{th:main3}
    If all Hurwitz-type matrices~$H(p_m,p_n)$ are totally nonnegative for $0\le n<m<M$ and at
    least one of them contains a nonzero minor of order two, then the
    series~$f(z)=\sum_{n=0}^{M-1}z^np_n(z^M)$ converges to a
    function~$g(z^M)\cdot q_1(z)\cdot q_2\left(\frac 1z\right)$ with no zeros in the
    sector~$C_M\coloneqq\big\{z\in\mathbb C:|\arg z|<\frac{\pi}{M}\big\}$, where both~$q_1(z)$
    and~$q_2(z)$ are entire functions of genus at most~$M-1$ and~$g(z)$ has the
    form~\eqref{eq:funct_gen_dtps}.
\end{theorem}
Already for polynomials, the absence of zeros in the sector~$C_M$ for~$M>2$ is necessary but not
sufficient for the total nonnegativity of the corresponding generalized Hurwitz matrix:
see~\cite[Example~35]{HoltzKushelKhrushchev16}. The question of sufficient conditions is opened.
Following~\cite{HoltzKushelKhrushchev16}, we only remark here that total nonnegativity of the
matrix~$(f_{jM-i})_{i,j=-\infty}^{\infty}$ implies total nonnegativity of its
submatrix~$(f_{jkM-i})_{i,j=-\infty}^{\infty}$ for each~$k=1,2,\dots$.

In the products and sums with inequalities in limits, we assume that the indexing variable
changes in~$\mathbb{Z}$ or in some finite or infinite subinterval of~$\mathbb{Z}$, and that it
additionally satisfies the indicated inequalities. Accordingly, a product or sum can be empty,
finite or infinite. By writing that a function has one of the above representations, we assume
that the involved products are locally uniformly convergent unless the converse is stated
explicitly. In the above theorems, the convergence follows from the total nonnegativity of the
involved matrices. The condition of convergence is well-known and can be expressed as the
following theorem, which we apply in the settings~$\zeta=z^{\pm 1}$ or~$\zeta=z^2$.
\begin{theorem}[see \emph{e.g.}~{\cite[pp.~7--13,~21]{Levin}}]\label{th:canonical_product}
    The infinite product~$\prod_{\nu=0}^\infty \left(1+\frac{\zeta}{\alpha_\nu}\right)$,
    converges uniformly in~$\zeta$ varying in compact subsets of~$\mathbb{C}$ if and only if the
    series~$\sum_{\nu=0}^\infty \frac 1{|\alpha_\nu|}$ converges. If so, then for
    any~$\varepsilon>0$ the estimates
    \( \prod_{\nu=0}^\infty \left|1+\frac{\zeta}{\alpha_\nu}\right| < C e^{\varepsilon R} \)
    and, outside exceptional disks with an arbitrarily small sum of radii,
    \( \prod_{\nu=0}^\infty \left|1+\frac{\zeta}{\alpha_\nu}\right| > C e^{-\varepsilon R} \)
    provided that~$|\zeta|\le R$ and the positive numbers~$R$ and~$C$ are big enough.
\end{theorem}
\setstretch{1.2}%
For example, by this theorem the convergence (locally uniform in some annulus centred at the
origin) of a series to a function of the form~\eqref{eq:f_weierstrass} implies the condition
\[
\sum_{\lambda\ne0}\frac 1{\xi_{\lambda}}+ \sum_{\nu\ne0}\frac1{|\Re\gamma_{\nu}|}
+ \sum_{\nu\ne0}\frac1{|\gamma_{\nu}|^{2}}<\infty.
\]

\begin{remark}
    The matrix~$H(p',p)$ is totally nonnegative if and only if~$p(z)$ represents an
    \emph{entire} function generating a totally positive sequence
    (see~\cite[Theorem~1.2]{Dyachenko14}). If so, then~$p'(z)$ also generates a totally positive
    sequence. Considering doubly infinite series does not change the picture: the Toeplitz
    matrix~$T(p')$ contains negative entries provided that the series~$p(z)$ has a positive
    coefficient at a negative power of~$z$. It can be shown that~$\frac{p(z)}{p'(z)}$ is not a
    mapping of~\Cp\ into itself (\emph{cf.} Theorem~\ref{th:main1}) provided that the
    function~$p(z)$ is not entire and generates a totally positive sequence. However, if~$p(z)$
    is meromorphic in~$\mathbb{C}\setminus\{0\}$, then the ratio~$\frac{zp'(z)}{p(z)}$ is a
    mapping of the upper half of the complex plane into itself exactly when~$p(z)$ generates a
    totally positive sequence. The further details can be found
    in~\cite[Chapter~5]{DyachenkoThesis} and~\cite{TyaglovDyachenko}.
\end{remark}

\section{\texorpdfstring{$\mathcal{S}$}{S}-functions}\label{sec:s-funct}
\begin{lemma}\label{lemma:product_is_S}
    The product
    \begin{equation}\label{eq:order_alphas_betas}
        \begin{gathered}
            C \frac{\prod_{\mu>0} \left(1+\dfrac{z}{\beta_\mu}\right)}
            {\prod_{\nu>0} \left(1+\dfrac{z}{\alpha_\nu}\right)}
            \frac{\prod_{\mu<0}\left(1+\dfrac{z^{-1}}{\beta_\mu}\right)}
            {\prod_{\nu<0}\left(1+\dfrac{z^{-1}}{\alpha_\nu}\right)},
            \quad\text{where $C>0$, and the numbers}
            \\
            0<\cdots<\alpha_{-2}^{-1}<\beta_{-1}^{-1}<\alpha_{-1}^{-1}
            <\beta_{1}<\alpha_{1}<\beta_{2}<\alpha_2<\cdots
        \end{gathered}
    \end{equation}
    satisfy~$\sum_{\nu\ne 0} \big(\alpha_\nu^{-1} + \beta_\nu^{-1}\big) <\infty$, determines
    an~$\mathcal S$-function. Analogously,
    \begin{equation}\label{eq:order_alphas_betas_mer}
        C
        \frac{z+\beta_0}{z+\alpha_0}\cdot
        \frac{\prod_{\mu>0} \left(1+\dfrac{z}{\beta_\mu}\right)}
        {\prod_{\nu>0} \left(1+\dfrac{z}{\alpha_\nu}\right)}
        ,
    \end{equation}
    where~$C\ge 0$ and the numbers~$0\le\beta_{0}<\alpha_{0}<\beta_{1}<\alpha_{1}<\cdots$
    satisfy~$\sum_{\nu>0} \big(\alpha_\nu^{-1} + \beta_\nu^{-1}\big) <\infty$ is a
    meromorphic $\mathcal S$\nobreakdash-function. Products over~$\mu$ and~$\nu$
    in~\eqref{eq:order_alphas_betas} or~\eqref{eq:order_alphas_betas_mer} can be terminating, in
    which case the numerator and the denominator retain to have interlacing zeros.
\end{lemma}
\begin{proof}
    Suppose that~$F(z)$ has the form~\eqref{eq:order_alphas_betas} and denote
    \begin{equation}\label{eq:F_is_limit}
        F_n(z)\coloneqq C\frac{q_n(z)}{p_n(z)},\ww
        q_n(z)=\prod_{\nu=1}^{n} \left(1+\frac{z}{\beta_\nu}\right)
        \left( 1+\frac{z^{-1}}{\beta_{-\nu}}\right)
        ,\quad
        p_n(z)=\prod_{\nu=1}^{n} \left(1+\frac{z}{\alpha_\nu}\right)
        \left( 1+\frac{z^{-1}}{\alpha_{-\nu}}\right).
    \end{equation}
    Note that the
    product~$\prod_{\nu=-n}^{-1}\frac{\alpha_\nu}{\beta_\nu}
    =\prod_{\nu=-n}^{-1}\frac{\beta_\nu^{-1}}{\alpha_\nu^{-1}}<1$ is bounded. For
    each~$n\in\mathbb{Z}_{>0}$, the rational function
    \begin{equation}\label{eq:F_is_limit_ML}
        F_n(z)
        =
        C\cdot
        \prod_{\nu=-n}^{-1}\frac{\alpha_\nu}{\beta_\nu}
        +\sum_{\nu=1}^n\left(\frac{A_{\nu,n} z}{z+\alpha_\nu}
            +\frac{A_{-\nu,n} z}{z+\frac{1}{\alpha_{-\nu}}}
        \right),
        \ww
        A_{\nu,n}=\left.C\frac{q_n(z)}{z p'_n(z)}\right|_{z=-\alpha_\nu^{\sign\nu}}>0,
    \end{equation}
    is an~$\mathcal S$-function as each of its partial fractions is such. The
    condition~$\sum_{\nu\ne 0} \big(\alpha_\nu^{-1} + \beta_\nu^{-1}\big) <\infty$ implies the
    locally uniform convergence of each product in~\eqref{eq:order_alphas_betas} (see
    Theorem~\ref{th:canonical_product}) and, therefore, of the numerator~$q_n(z)$ and the
    denominator~$p_n(z)$ as~$n\to\infty$. Since the denominator is nonzero for~$z\not\le0$, the
    function~$F(z)$ is the limit of~$F_n(z)$ as~$n\to\infty$ uniform on compact subsets
    of~$\mathbb{C}\setminus(-\infty,0]$. Moreover,
    \setstretch{1.1}%
    \[\Im F(z)\cdot\Im z=\lim_{n\to\infty}F_n(z)\cdot\Im z\ge 0;\]
    the inequality is strict outside the real line due to the maximum principle for the harmonic
    function~$\Im F(z)$.

    The assertion that the expression~\eqref{eq:order_alphas_betas_mer} represents an
    $\mathcal S$-function follows by omitting from~\eqref{eq:F_is_limit_ML} terms that
    correspond to absent poles.
\end{proof}
\setstretch{1.1}%
\begin{lemma}\label{lemma:prop_S_exp}
    Let~$F(z)$ be a function of the form~\eqref{eq:order_alphas_betas}
    or~\eqref{eq:order_alphas_betas_mer} and let real numbers~$A,A_0,p$ be such
    that~$A^2+A_0^2 > 0$. Then~$G(z)\coloneqq e^{Az+\frac{A_0}z}z^pF(z)$ is not a mapping
    of~\Cp\ into itself and~$G(z_0)<0$ for some~$z_0\notin\mathbb{R}$.
\end{lemma}
\begin{proof}
    Denote the multivalued argument function by~$\Arg$ and its principal value by~$\arg$. In the
    special case when~$F(z)$ is a real constant~$\Arg F(z)$ does not depend on~$z$.
    Otherwise~$0<\arg F(z)<\pi$ wherever~$\Im z>0$ by Lemma~\ref{lemma:product_is_S}, and
    therefore~$|\Arg F(ir_1)-\Arg F(i r_2)|<2\pi$ for any~$r_2>r_1>0$.
    Furthermore,~$\Arg(ir_1)^p=\Arg(ir_2)^p$
    and~$\Arg e^{iAr-i\frac{A_0}{r}}=\frac 1i\Ln e^{iAr-i\frac{A_0}{r}}$, where~$\Ln$ is the
    multivalued logarithm and~$r>0$. Branches of the logarithm differ by a constant, thus
    \[
        \frac {d}{dr}\Arg e^{iAr-i\frac{A_0}{r}}
        =\frac 1i\left(iAr-i\frac{A_0}{r}\right)'=A+\frac{A_0}{r^2}
    \]
    and (due to~$A^2+A_0^2>0$) the integral of this expression over~$(r_1,r_2)$ can be made
    arbitrarily big in absolute value through the choice of positive numbers~$r_1$ and~$r_2$.
    More specifically, we always can chose~$r_2>r_1>0$ so that
    \begin{gather*}
        \left|
            \Arg e^{iAr_2-i\frac{A_0}{r_2}}-\Arg e^{iAr_1-i\frac{A_0}{r_1}}
        \right|
        = \left|
            \int_{r_1}^{r_2} \frac {d}{dr}\Arg e^{iAr-i\frac{A_0}{r}}\, dr
        \right|
        \ge 4\pi,
        \quad\text{and hence}\\
        \left|
            \Arg G(ir_2)-\Arg G(i r_1)
        \right|
        > 
        \left|
            \Arg e^{iAr_2-i\frac{A_0}{r_2}}-\Arg e^{iAr_1-i\frac{A_0}{r_1}}
        \right|
        -\left|
            \Arg F(ir_2)-\Arg F(i r_1)
        \right|
        >2\pi.
    \end{gather*}
    In other words, the interval~$(r_1,r_2)$ contains at least one point~$r$ such
    that~$\arg G(ir)=\pi$, that is~$G(z_0)<0$ with~$z_0=ir$.
\end{proof}
\begin{lemma}\label{lemma:prop_S}
    Given~$F(z)$ of the form~\eqref{eq:order_alphas_betas} or~\eqref{eq:order_alphas_betas_mer}
    not equal identically to~$Cz$ or~$C$, the function~$z^pF(z)$ with any real~$p\notin(-1,0]$
    and~$A,A_0\in\mathbb{R}$ cannot be a mapping of~$\mathbb{C}_+$ into itself; the
    function~$\frac{z}{F(z)}$ is an $\mathcal{S}$-function of the
    form~\eqref{eq:order_alphas_betas} or~\eqref{eq:order_alphas_betas_mer}. Moreover, under the
    additional condition~$p\ne-1$ there exist a point~$z_0\notin\mathbb{R}$ such
    that~$z_0^pF(z_0)<0$.
\end{lemma}
\begin{proof}
    The reciprocal of the product~\eqref{eq:order_alphas_betas} can be expressed as
    \begin{equation}\label{eq:recipr_to_S}
        \frac 1{F(z)}= C \frac
        {\prod_{\nu>0} \left(1+\frac{z}{\alpha_\nu}\right)}
        {\prod_{\mu>0} \left(1+\frac{z}{\beta_\mu}\right)}
        \frac{\prod_{\nu<0} \left(1+\frac{z^{-1}}{\alpha_\nu}\right)}
        {\prod_{\mu<0}\left(1+\frac{z^{-1}}{\beta_\mu}\right)}
        =
        \frac C{z\alpha_{-1}} \frac
        {\left(z\alpha_{-1}+1\right)\prod_{\nu>0} \left(1+\frac{z}{\alpha_\nu}\right)}
        {\prod_{\mu>0} \left(1+\frac{z}{\beta_\mu}\right)}
        \frac{\prod_{\nu<-1}\left(1+\frac{z^{-1}}{\alpha_\nu}\right)}
        {\prod_{\mu<0} \left(1+\frac{z^{-1}}{\beta_\mu}\right)}.
    \end{equation}
    Therefore, relabelling the~$\beta_\mu\mapsto \widetilde\alpha_\mu$ for all~$\mu\ne0$;
    $\alpha_{-1}\mapsto\widetilde\beta_1$ and~$\alpha_\nu\mapsto \widetilde\beta_{\nu-1}$ for
    all~$\nu\notin\{0,1\}$ yields that~$\frac z{F(z)}$ has the
    form~\eqref{eq:order_alphas_betas}. An analogous reasoning works for the reciprocal
    of~\eqref{eq:order_alphas_betas_mer}.

    Now, let~$p>0$. Non-constant functions of the form~\eqref{eq:order_alphas_betas}
    or~\eqref{eq:order_alphas_betas_mer} have at least one negative simple zero. Therefore,
    there exists~$r>0$ such that~$F(-r)<0$. On the semicircle~$\{z\in\mathbb{C}_+:|z|=r\}$, we
    have the following conditions:
    \[
        0<\Arg F(z)-\Arg F(r)<
        \Arg F(-r)-\Arg F(r)=\pi
        \an
        \Arg z^p-\Arg r^p = p\arg z<p\pi.
    \]
    The above inequalities yield that
    \[
        \Arg (-r)^pF(-r)-\Arg r^pF(r) = p\pi+\pi>\pi,
    \]
    \setstretch{1.2}%
    so the increment~$\Arg z^pF(z)-\Arg z^pF(z)$ equals to~$\pi$ at least at one
    point~$z_0\in\mathbb{C}_+$ satisfying~$|z_0|=r$. In particular,~$F(z_0)< 0$. If~$p<-1$, then
    the previous reasoning implies that the function~$z^{-p-1}\frac z{F(z)}$ is negative at some
    point~$z_0$ of the upper half-plane; therefore,~$z_0^{p}F(z)=z_0^{p+1}\frac{F(z_0)}{z_0}< 0$
    as well.

    The remaining case of~$p=-1$ follows from the identity
    \(
        \frac {F(z)}z = \overline {\left(\frac z{F(z)}\right)} \cdot \frac {|F(z)|^2}{|z|^2}
    \) because~$\Im\frac z{F(z)}>0$ wherever~$\Im z>0$.
\end{proof}
\setstretch{1.2}%
\begin{proof}[Proof of Lemma~\ref{lemma:prop_S1}]
    On account of Lemma~\ref{lemma:product_is_S}, it is enough to prove that the function~$F(z)$
    is an $\mathcal S$-function only if has the form~\eqref{eq:order_alphas_betas}
    or~\eqref{eq:order_alphas_betas_mer}. Since~$F(z)$ is regular for~$z>0$, the poles of~$p(z)$
    and~$q(z)$ coincide and have the same orders. Therefore, the function~$F(z)$ is positive
    when~$z>0$. As is shown in Lemma~\ref{lemma:prop_S_exp},~$F(z)$ does not have exponential
    factors; that is, the exponential factors in the representations~\eqref{eq:funct_gen_dtps}
    of the functions~$p(z)$ and~$q(z)$ coincide. As a result, $F(z)=z^pG(z)$, where~$p$ is
    integer and~$G(z)$ is an ~$\mathcal{S}$-function. The exponent~$p$ must then be zero by
    Lemma~\ref{lemma:prop_S}.
\end{proof}
\section{Total nonnegativity and interlacing zeros}
Assume in this section that~$p(z)\coloneqq\sum_{k=-\infty}^\infty\,a_kz^k$,
~$q(z)\coloneqq\sum_{k=-\infty}^\infty\,b_kz^k$
and~$\widetilde p(z)\coloneqq zp(z)=\sum_{k=-\infty}^\infty\,a_kz^{k+1}$.
\begin{lemma}\label{lemma:T_via_HH}
    If the matrix~$H(p,q)$ is totally nonnegative
    , then for arbitrarily taken nonnegative numbers~$A$ and~$B$ both matrices~$T(Ap+Bq)$
    and~$T(Aq+B\widetilde p)$ are totally nonnegative.
\end{lemma}
\begin{proof}
    The matrices~$H(p,q)$ and
    \[
        H(q,\widetilde p)={\begin{pmatrix}
                \ddots & \vdots & \vdots & \vdots & \vdots & \vdots & \vdots & \iddots\\
                \hdots & b_0& b_1& b_2 & b_3 & b_4 & b_5 & \hdots\\
                \hdots & a_{-1} & a_0& a_1& a_2 & a_3 & a_4 & \hdots\\
                \hdots & b_{-1} & b_0 & b_1 & b_2 & b_3 & b_4 & \hdots\\
                \hdots & a_{-2} & a_{-1} & a_0 & a_1 & a_2 & a_3 & \hdots\\
                \hdots & b_{-2} & b_{-1} & b_0 & b_1 & b_2 & b_3 & \hdots\\
                \hdots & a_{-3} & a_{-2} & a_{-1} & a_0 & a_1 & a_2 & \hdots\\
                \iddots & \vdots & \vdots & \vdots & \vdots & \vdots & \vdots & \ddots
            \end{pmatrix}}.
    \]
    coincide up to a shift in indexation; that is,~$H(q,\widetilde p)$ can be obtained by
    increasing the indices of rows in~$H(p,q)$ by~$1$. In particular, the
    matrix~$H(q,\widetilde p)$ is totally nonnegative.

    Observe that
    \begin{equation}\label{eq:T_via_HH}
        T(Ap+Bq) = H^{\textsf{T}}(A,B)\,H(p,q)
        \an
        T(Aq+B\widetilde p) = H^{\textsf{T}}(A,B)\,H(q,\widetilde p),
    \end{equation}
    where the auxiliary totally nonnegative matrix~$H^{\textsf{T}}(A,B)$ is the transpose
    of~$H(A,B)$:
    \[
    H^{\textsf{T}}(A,B)=
    \begin{pmatrix}
        \ddots & \vdots & \vdots & \vdots & \vdots & \vdots & \vdots & \iddots\\
        \hdots & A & B & 0 & 0 & 0 & 0 & \hdots\\
        \hdots & 0 & 0 & A & B & 0 & 0 & \hdots\\
        \hdots & 0 & 0 & 0 & 0 & A & B & \hdots\\
        \iddots & \vdots & \vdots & \vdots & \vdots & \vdots & \vdots & \ddots
    \end{pmatrix}
    =\big(h_{ij}\big)_{i,j=-\infty}^{\infty},\ww
    h_{ij}=\begin{cases}
        A&\text{if } j=2i-1,\\
        B&\text{if } j=2i,\\
        0&\text{otherwise}.
    \end{cases}
    \]
    Therefore, applying the Cauchy-Binet formula to the expressions~\eqref{eq:T_via_HH} yields
    that all minors of the matri\-ces~$T(Ap+Bq)$ and~$T(Aq+B\widetilde p)$ must be nonnegative.
\end{proof}

\setstretch{1.2}%
\begin{lemma}\label{lemma:TT_TNN_conv}
    Let the matrix~$T(Ap+Bq)$ be totally nonnegative for every choice of the numbers~$A,B\ge 0$.
    If~$T(p)$ has a nonzero minor of order~$2$, then there exists a nonempty
    annulus~$\mathfrak A$ centred at the origin where both
    series~$\phi(z;A,B)\coloneqq Ap(z)+Bq(z)$ and~$\psi(z;A,B)\colonequals Aq(z)+Bzp(z)$ with
    any choice of~$A,B\ge 0$ converge absolutely. If all minors of~$T(p)$ of order~$2$ are equal
    to zero, then~$b_0p(z)=a_0q(z)$ coefficient-wise.
\end{lemma}
\begin{proof}
    Given a real number~$a$ let~$\lfloor a \rfloor$ denote the maximal integer less then or
    equal to~$a$. The straightforward consequence of the total nonnegativity of~$T(p)$ is
    \( a_{m+1}a_{n}\le a_{n+1}a_{m} \) for all indices~$n<m$. If~$a_na_{k}>0$ for some~$k>n$,
    then we therefore
    have~\(0<a_na_k\le a_{n+1}a_{k-1}\le\cdots\le
    a_{\left\lfloor\frac{n+k}2\right\rfloor}a_{\left\lfloor\frac{n+k+1}2\right\rfloor}\), which
    implies~$a_m>0$ whenever~$m=n+1,n+2,\dots,k-1$; in other words, the series~$p(z)$ has no gaps.
    Moreover,
    \[
    0\le\frac{a_{n}}{a_{n+1}}\le\frac{a_{m}}{a_{m+1}}
    \]
    provided that~$n<m$ and the denominators are nonzero. Accordingly,~$p(z)$ converges in the
    annulus~$0\le r<z<R\le+\infty$ by the ratio test (unless~$r=R$), where
    \[
    r\coloneqq
    \begin{cases}
        \displaystyle
        \lim_{k\to-\infty}\frac{a_k}{a_{k+1}},&\text{if }a_k>0\text{ for all $-k$ big enough}\\
        0,& \text{otherwise}
    \end{cases}
    \le
    R\coloneqq
    \begin{cases}
        \displaystyle
        \lim_{k\to+\infty}\frac{a_k}{a_{k+1}},&\text{if }a_k>0\text{ for all $k$ big enough}\\
        +\infty,& \text{otherwise}
    \end{cases}
    \!.
    \]
    Since the matrix~$T(p)$ has a nonzero minor of order two, Theorem~\ref{th:E-AESW} implies
    that the annulus~$r<z<R$ for the series~$p(z)$ is not empty.
    
    Suppose that the series~$q(z)$ has an empty annulus of convergence; then
    Theorem~\ref{th:E-AESW} implies~$b_n^2=b_{n+1}b_{n-1}\ne 0$ for any integer~$n$. The
    estimate
    \begin{equation*}
        0\le
        \begin{vmatrix}
            a_n+B b_n&a_{n+1}+B b_{n+1}\\
            a_{n-1}+B b_{n-1}&a_n+B b_n
        \end{vmatrix}
        = (2a_nb_n-a_{n+1}b_{n-1}-a_{n-1}b_{n+1}) B + a_n^2-a_{n-1}a_{n+1}
    \end{equation*}
    holds true for every~$B>0$, and hence~$0\le 2a_nb_n-a_{n+1}b_{n-1}-a_{n-1}b_{n+1}$. Since
    the ratio~$C\coloneqq\frac{b_{n-1}}{b_n}$ is independent of~$n$, the inequality
    \begin{equation}\label{eq:ineq_coef_pq}
        0\le 2a_n\frac{b_{n}}{b_{n+1}}-a_{n+1}\frac{b_{n-1}}{b_{n+1}}-a_{n-1}
        = -a_{n+1}C^2 + 2a_n C-a_{n-1}
    \end{equation}
    must be satisfied for each~$n$. Nevertheless, the condition~$a_n=0\ne a_{n-1}+a_{n+1}$ for
    some~$n$ gives the contradiction~$0\le -a_{n+1}C^2-a_{n-1}<0$. Thus, all coefficients of the
    series~$p(z)$ are nonzero and the estimate~\eqref{eq:ineq_coef_pq} is equivalent to
    \[
    \left(C-\frac{a_n}{a_{n+1}}\right)^2\le
    \left(\frac{a_n}{a_{n+1}}\right)^2-\frac{a_{n-1}}{a_{n+1}} =
    \left(\frac{a_n}{a_{n+1}}-\frac{a_{n-1}}{a_{n}}\right)\frac{a_{n}}{a_{n+1}}.
    \]
    Taking limits as~$n\to\pm\infty$ then yields the equality~\(r=C=R\).

    Suppose that~$T(p)$ has a nonzero minor of order~$2$. Theorem~\ref{th:E-AESW} then implies
    that the series~$p(z)$ converges; that is we have the contradiction~$r<R$ unless~$q(z)$
    converges in some annulus. Analogously, the series~$\phi(z;1,1)=p(z)+q(z)$ has a nonempty
    annulus~$\mathfrak A$ of convergence as well (otherwise the above reasoning
    for~$\phi(z;1,1)$ instead of~$q(z)$ would imply that~$r=R$). Moreover, both
    series~$\phi(z;A,B)= Ap(z)+Bq(z)$ and~$\psi(z;A,B)= Aq(z)+Bzp(z)$ with any choice
    of~$A,B\ge 0$ are absolutely convergent in~$\mathfrak A$ since all coefficients of the
    involved series are nonnegative.

    Suppose that~$T(p)$ has no nonzero minors of order~$2$. If~$p(z)\equiv 0$ or~$q(z)\equiv 0$,
    then~$b_0p(z)\equiv a_0q(z)\equiv 0$ which implies the lemma in this case. Otherwise, the
    series~$p(z)$ diverges by Theorem~\ref{th:E-AESW} and~$r=R=\frac{a_n}{a_{n+1}}$ for all~$n$.
    Thus, the above part of the proof with the exchanged roles of~$p(z)$ and~$q(z)$ yields
    that~$\frac{b_n}{b_{n+1}}=r$ whenever~$n\in\mathbb Z$. Consequently, the
    equality~$a_0q(z)=b_0p(z)$ is satisfied in the sense of formal power series, \emph{i.e.}
    coefficient-wise.
\end{proof}
\begin{corollary}\label{cor:total-nonn-interl}
    If the totally nonnegative matrix~$H(p,q)$ has a nonzero minor of order~$2$
    and~$p(z)\not\equiv 0$, then~$T(p)$ has a nonzero minor of order~$2$ as well.
\end{corollary}
\begin{proof}
    Indeed, Lemma~\ref{lemma:T_via_HH} implies that all minors of the matrix~$T(Ap+Bq)$ are
    nonnegative for every choice of the numbers~$A,B\ge 0$, so we can apply
    Lemma~\ref{lemma:TT_TNN_conv}. We have two possibilities: the first is that the entries
    involved in the nonzero minor of~$H(p,q)$ only come from one of the involved series. That
    is, this minor is actually a minor of~$T(p)$, or of~$T(q)$ and hence~$T(p)$ also has a
    nonzero minor of order~$2$ by Lemma~\ref{lemma:TT_TNN_conv}. Another possibility is that the
    terms of the nonzero minor come from both involved series: $a_{m}b_{n+k}>a_{m+k}b_{n}$
    \;or\;~$a_{n+k}b_{m+1}>a_{n}b_{m+1+k}$ for some integers~$k>0$ and~$m\ge n$, so
    automatically~$p(z)\not\equiv 0$ and~$q(z)\not\equiv 0$. In this case, the assumption that
    all minors of~$T(p)$ or~$T(q)$ or order~$2$ are zero yields a contradiction due
    to~$\frac{a_{m}}{a_{m+k}}=\frac{a_{n}}{a_{n+k}}=\frac{b_{m+1}}{b_{m+1+k}}=\frac{b_{n}}{b_{n+k}}=\textit{const}$
    by virtue of Lemma~\ref{lemma:TT_TNN_conv}.
\end{proof}

\begin{lemma}\label{lemma:H_TNN_F_has_form}
    Let the matrices~$T(Ap+Bq)$ and~$T(Aq+B\widetilde p)$ be totally nonnegative for
    all~$A,B\ge 0$, and let~$T(p)$ have a nonzero minor of order~$2$. Then the
    ratio~$F(z)\coloneqq\frac{q(z)}{p(z)}$ is an $\mathcal{S}$-function.
\end{lemma}
\begin{proof}
    By Lemma~\ref{lemma:TT_TNN_conv}, the series~$\phi(z;A,B)=Ap(z)+Bq(z)$
    and~$\psi(z;A,B)=Aq(z)+Bzp(z)$ converge in some common annulus. The Toeplitz matrices
    constructed from the coefficients of~$\phi(z;A,B)$ and~$\psi(z;A,B)$ are totally
    nonnegative, and hence the analytic continuations of~$\phi(z;A,B)$ and~$\psi(z;A,B)$ have
    the form~\eqref{eq:funct_gen_dtps}. In particular, all zeros of the functions~$\phi(z;A,B)$
    and~$\psi(z;A,B)$ lie in~$(-\infty,0]$. If~$z_0$ is such
    that~$F(z_0)=\frac{q(z_0)}{p(z_0)}\le 0$,
    then~$\phi(z_0;-F(z_0),1)=-\frac{q(z_0)}{p(z_0)}p(z_0)+q(z_0)=0$. Since for each~$A\ge 0$
    the function~$\phi(z;A,1)$ does not vanish outside~$(-\infty,0]$, the inequality~$z_0\le 0$
    must be true. Analogously, if~$z_1$ is such
    that~$\frac{z_1}{F(z_1)}=\frac{z_1p(z_1)}{q(z_1)}\le 0$,
    then~$\psi\left(z_1;-\frac{z_1}{F(z_1)},1\right)=0$. Since for each~$A\ge 0$ the
    function~$\psi(z;A,1)$ is nonzero outside~$(-\infty,0]$, we obtain~$z_1\le 0$. In
    particular, the function~$F(z)$ has no positive poles or zeros; thus, it is positive and
    holomorphic in~$(0,+\infty)$. In other words, all positive poles of the functions~$p(z)$
    and~$q(z)$ coincide with orders.
    
    \begin{fact}[{Details can be found in \emph{e.g.}~\cite[p.~19]{Duren2004}}]\label{fact:A}
        Let~$h(z)$ be a real function holomorphic in a neighbourhood of a real point $x$ and
        such that~$h(z)\le 0$ for a complex~$z$ implies~$z\le 0$. Then the
        expression~$h(z)-h(x)$ has a zero at~$x$ of some multiplicity~$r\ge 1$.
        Therefore,~$h(z)-h(x)\sim (z-x)^{r}$ as~$z$ is close to~$x$ in a small enough
        neighbourhood of~$x$, and we have~$\Im h(z)=0$ on the union of~$r$ arcs meeting in this
        neighbourhood only at~$x$; one of these arcs is a subinterval of the real line due to
        the reality of~$h(z)$. Furthermore, the half of (if~$r$ is even) or all (if~$r$ is odd)
        the arcs contain an interval where~$h(z)\le h(x)$.
        \emph{In particular, the condition~$h(x)=0$ implies~$r\le 2$, and the condition~$h(x)<0$
            implies~$r=1$.}
    \end{fact}
    Fact~\ref{fact:A} with~$h(z)\coloneqq F(z)$ implies that~$F(z)$ can have at most double
    zeros. It is possible that the function~$F(z)$ is holomorphic at the origin and equal to
    zero there. The assumption that the point~$x=0$ can be a double zero of~$F(z)$ is
    contradictory: Fact~\ref{fact:A} implies that~$F(z)$ is negative for all real~$z\ne0$ small
    enough, which is impossible for~$z>0$. Suppose that~$x<0$ is a double zero of~$F(z)$, that
    is~$F(x)=F'(x)=0\ne F''(x)$. Then~$F(z)<0$ and, therefore,~$z^{-1}F(z)>0$ for all real~$z$
    in a sufficiently small punctured neighbourhood of~$x$. At the point~$x$, the
    function~$z^{-1}F(z)$ has a double zero:
    \[
    \frac{F(x)}x=\frac{F(x)-xF'(x)}{x^2}=0\ne \frac{2x(F(x)-xF'(x))-x^{3}F''(x)}{x^{4}}.
    \]
    Putting~$h(z)\coloneqq z^{-1}F(z)$ in Fact~\ref{fact:A} then yields a contradiction, since
    the inequality~$z^{-1}F(z)\le0$ must be satisfied for all real~$z$ which are close enough
    to~$x$. Consequently, the only possible case is~$r=1$, that is that all zeros of~$F(z)$ are
    simple. Considering in the same way~$h(z)=\frac{z}{F(z)}$ and~$h(z)=\frac{1}{F(z)}$ shows
    that all poles of~$\frac{F(z)}{z}$ are simple. In particular,~$F(z)$ cannot have a pole at
    the origin.
    
    Now, let us prove that zeros and poles of~$F(z)$ are interlacing. Suppose
    that~$x_1<x_2\le 0$ are two consecutive zeros of the function~$F(z)$, such that the
    interval~$(x_1,x_2)$ contains no poles of~$F(z)$. The ratio~$z^{-1}F(z)$ also vanishes
    at~$x_1$ and~$x_2$ unless~$x_2=0$; therefore, Rolle's theorem gives the
    points~$\xi_1,\xi_2\in(x_1,x_2)$ such
    that~$F'(\xi_1)=\xi_2^{-2}(F(\xi_2)-\xi_2F'(\xi_2))=0$. Let~$h(z)\coloneqq F(z)$
    and~$x\coloneqq\xi_1$ if~$F(\xi_1)<0$, or~$h(z)\coloneqq z^{-1}F(z)$ and~$x\coloneqq\xi_2$
    if~$F(\xi_1)>0$. In the special case~$x_2=0$, the function~$F(z)$ is negative
    in~$(x_1,x_2)$, so we put~$h(z)\coloneqq F(z)$ and denote a zero of~$h'(z)$ in this interval
    by~$x$. Fact~\ref{fact:A} implies~$h'(z)\ne 0$ in the whole interval~$x_1<z<x_2\le 0$
    including~$z=x$, which contradicts to our choice of~$x$. This shows that the function~$F(z)$
    has at least one pole between each pair of its zeros. The same argumentation
    for~$\frac {z}{F(z)}$ instead of~$F(z)$ yields that~$F(z)$ has a zero between each pair of
    its poles. As a result, zeros and poles of~$F(z)$ are interlacing.

    Recall that the functions~$p(z)$ and~$q(z)$ can be represented as
    in~\eqref{eq:funct_gen_dtps} and that their poles coincide with orders. These functions
    cannot have distinct exponential factors: otherwise~$F(z)$ gets the corresponding
    exponential factor, so Lemma~\ref{lemma:prop_S_exp} implies that~$F(z_0)<0$ for some~$z_0$
    outside the real line. Therefore,~$F(z)=z^pG(z)$, where~$p$ is an integer and~$G(z)$ has the
    form~\eqref{eq:order_alphas_betas} or~\eqref{eq:order_alphas_betas_mer}. The case
    when~$F(z)$ has the form~$Cz^r$ with~$r\in\mathbb{Z}$ yields~$r=1$ or~$r=0$ because~$F(z)$
    has no poles at the origin and the possible zero at the origin can only be simple. In the
    case~$F(z)\ne Cz^r$, Lemma~\ref{lemma:prop_S} implies~$p=0$ since both functions~$F(z)$
    and~$z^{-1}F(z)$ can attain negative values only on the real line.
\end{proof}
\begin{lemma}\label{lemma:converse}
    If functions~$p(z)$ and~$q(z)$ have the form~\eqref{eq:funct_gen_dtps} and their
    ratio~$F(z)=\frac{q(z)}{p(z)}$ can be represented as in~\eqref{eq:order_alphas_betas}
    or~\eqref{eq:order_alphas_betas_mer}, then the matrix~$H(p,q)$ is totally nonnegative.
\end{lemma}
\begin{proof}
    Indeed, denote by~$p_*(z)\coloneqq \frac{p(z)}{g(z)}\not\equiv 0$
    and~$q_*(z)\coloneqq \frac{q(z)}{g(z)}$ the denominator and numerator of the function~$F(z)$
    given in~\eqref{eq:order_alphas_betas}. This means that~$p_*(z)$ and~$q_*(z)$ have no common
    zeros, no poles and no exponential factors; the function~$g(z)$ has the
    form~\eqref{eq:funct_gen_dtps}. The function~$C\frac{q_n(z)}{p_n(z)}=F_n(z)$ introduced
    in~\eqref{eq:F_is_limit} maps the upper half-plane into itself for each positive
    integer~$n$. According to Theorem~3.44 of~\cite{HoltzTyaglov} (see also Theorem~1.4
    of~\cite{Dyachenko14} where the notation is closer to the current paper) the
    matrix~$H(p_n,q_n)$ is totally nonnegative. Since~$p_n(z)$ and~$q_n(z)$ converge
    in~$\mathbb C\setminus\{0\}$ locally uniformly to~$p_*(z)$ and~$q_*(z)$ respectively, their
    Laurent coefficients converge as well. Therefore, the matrix~$H(p_*,q_*)$ is totally
    nonnegative as an entry-wise limit of totally nonnegative matrices. Then the Cauchy-Binet
    formula implies the total nonnegativity of the
    matrix~$H(p,q)=H(p_*\cdot g,q_*\cdot g)=H(p_*,q_*)\cdot T(g)$, because~$T(g)$ is totally
    nonnegative by Theorem~\ref{th:E-AESW}.
\end{proof}
\begin{proof}[Proof of Theorem~\ref{th:main1}]
    Lemma~\ref{lemma:prop_S1} shows that the
    implications~\eqref{item:m2}$\implies$\eqref{item:m1}
    and~\eqref{item:m1}$\implies$\eqref{item:m3} follow, respectively, from
    Lemma~\ref{lemma:H_TNN_F_has_form} and Lemma~\ref{lemma:converse}. The
    implication~\eqref{item:m3}$\implies$\eqref{item:m2} follows from Lemma~\ref{lemma:T_via_HH}
    and Corollary~\ref{cor:total-nonn-interl}.
\end{proof}

\section{Proofs of Theorems~\ref{th:main2} and~\ref{th:main3}}
There are well-known relations between stable entire functions (more specifically, strongly
stable --- of the class~$\mathcal{HB}$ up to a change of the variable) and mappings of the upper
half of the complex plane into itself, see \emph{e.g.}~\cite{ChebMei,Levin}. In this section, we
adapt these relations to suit our problem; some simplifications arise since we only consider the
real case.

\begin{lemma}[{\emph{cf.}~\cite[pp.~307--308]{Levin}}]\label{lemma:h_lhp_less_h_rhp}
    Let~$q(z^2)+zp(z^2)$ be a non-trivial two-way infinite series. If the Hurwitz
    matrix~\(H(p,q)\) is totally nonnegative and has a nonzero minor of order two, then this
    series converges in some annulus to a function~$g(z^2)h(z)$, where~$g(z)$ generates
    a totally positive sequence,~$h(z)$ is holomorphic for~$z\ne 0$ and (unless it is equal
    identically to a constant) satisfies $|h(z)|> |h(-z)|$ wherever~$\Re z>0$.
\end{lemma}
Note that the statement of this lemma implies that the function~$h(z)$ is real
(\emph{i.e}~$h(\overline z)=h(z)$ for all~$z$) and that the function~$g(z^2)$ can only have real
poles and purely imaginary zeros. The converse of Lemma~\ref{lemma:h_lhp_less_h_rhp} to be true
requires a more delicate characterization of the function~$h(z)$, which is introduced in
Theorem~\ref{th:main2}.
\begin{proof} By Theorem~\ref{th:main1}, total nonnegativity of~$H(p,q)$ implies that both
    $p(z)$ and~$q(z)$ are of the form~\eqref{eq:funct_gen_dtps} and their
    ratio~$\frac {q(z)}{p(z)}$ is an $\mathcal S$-function. By Lemma~\ref{lemma:prop_S}, the
    ratio~$\frac {zp(z)}{q(z)}$ is an $\mathcal S$-function as well. If~$\arg$ denotes
    the principle branch of the argument, then the implication
    \[
    0\le\arg z<\frac \pi 2
    \implies
    -\pi< \arg \frac {p(z^2)}{q(z^2)} \le 0
    \an
    0\le \arg \frac {z^2p(z^2)}{q(z^2)} < \pi
    \]
    yields that the product~$\frac {p(z^2)}{q(z^2)}\cdot\frac {z^2p(z^2)}{q(z^2)}$ cannot be
    negative; the product cannot be zero or infinite by the definition
    of~$\mathcal{S}$-functions. Therefore, the principal value of its square root satisfies
    \[
    -\frac{\pi}2<
    \arg \sqrt{\frac {z^2p^2(z^2)}{q^2(z^2)}}
    =
    w(z)
    < \frac \pi 2,
    \ww
    w(z)
    \coloneqq
    \frac {zp(z^2)}{q(z^2)},
    \]
    on condition that~$0\le\arg z<\frac \pi 2$; the same inequality for~$-\frac \pi 2<\arg z< 0$
    follows by complex conjugation. In other words, the function~$w(z)$ maps the right half of
    the complex plane into itself. (Up to a change of the variable, we got an adaptation
    of~\cite[Lemma~SI.5.1]{KreinKac}.) Taking a linear-fractional transform of the right
    half-plane onto the unit disk gives
    \[
    1>
    \left|\frac {1-w(z)}{1+w(z)}\right|
    =
    \left|\frac {q(z^2)-zp(z^2)}{q(z^2)+zp(z^2)}\right|
    =
    \left|\frac {h(-z)}{h(z)}\right|
    \quad
    \text{as}
    \quad
    \Re z> 0
    .
    \]
\end{proof}
If~$h(z)$ is analytic in~$\mathbb{C}\setminus\{0\}$ and satisfies to~$|h(z)|>|h(-z)|$
as~$\Re z>0$, then~$h(z)$ clearly have no zeros with positive real parts. This conclusion can be
strengthened with the help of the Carleman formula.
\begin{lemma}\label{lemma:h_is_A}
    Suppose that~$h(z)$ is analytic in~$\mathbb{C}\setminus\{0\}$ and satisfies~$|h(z)|>|h(-z)|$
    as~$\Re z>0$. If~$I$ is a subinterval of\/~$\mathbb{Z}$ and~$(\gamma_\nu)_{\nu\in I}$ is
    a sequence of all zeros of~$h(z)$ ordered so
    that~$\nu,\nu+1\in I\implies\Re\gamma_\nu\ge\Re\gamma_{\nu+1}$ (counting with
    multiplicities), then
    \[
    \sum_{\nu\in I,\ \nu\ge0}\left|\Re\frac{1}{\gamma_\nu}\right|
    +
    \sum_{\nu\in I,\ \nu<0}\left|\Re\gamma_\nu\right|<\infty.
    \]
\end{lemma}
\begin{proof}
    Let $0<\xi<R$ be such that a function~$f(z)$ is analytic in the
    semi-annulus~$\xi\le|z|\le R$, \ $\Im z\ge 0$ and nonzero on its boundary. The Carleman
    formula (see e.g.~\cite[p.~224]{Levin} or~\cite[p.~153]{ChebMei}) for the function~$f(z)$
    can be written as
    \[
    \sum_{\nu\in J} \left(\frac 1{r_\nu} - \frac{r_\nu}{R^2}\right)\sin\theta_\nu
    = \frac 1{\pi R}\int_{0}^\pi \ln|f(Re^{i\theta})|\sin\theta d\theta
    + \frac 1{\pi R}\int_{\xi}^R \left(\frac 1{x^2}-\frac{1}{R^2}\right) \ln|f(x)f(-x)|dx
    + A_{\xi,R},    
    \]
    where~$\{r_\nu e^{i\theta_\nu}\}_{\nu\in J}$ is the set of all zeros counted with multiplicities,
    which~$f(z)$ has in this semi-annulus; the number~$A_{\xi,R}$ is defined as
    \[
    A_{\xi,R}=\Im \frac 1{\pi R}\int_{\pi}^0 \ln f(\xi e^{i\theta})\left(
        \frac 1{R^2} - \frac{e^{-2i\theta}}{\xi^2}\right)\xi e^{i\theta} d\theta.
    \]
    Note that~$A_{\xi,R}$ remains bounded if~$R$ grows to infinity. So, if~$f(z)$ is
    bounded, then putting~$R\to+\infty$ in the Carleman formula yields that the series
    \[
    \sum_{r_\nu>\xi} \left|\frac {\sin\theta_\nu}{r_\nu} - \frac{r_\nu\sin\theta_\nu}{R^2}\right|
    \quad
    \text{and, hence,}
    \quad
    \sum_{r_\nu>\xi} \left|\frac {\sin\theta_\nu}{r_\nu}\right|=
    \sum_{r_\nu>\xi} \left|\Im\frac {1}{r_\nu e^{i\theta_\nu}}\right|
    \]
    must be convergent.

    Chose~$\xi$ so that~$|\gamma_\nu|\ne\xi$ for all~$\nu\in I$. Within the settings
    \( f(iz)\coloneqq \frac{h(-z)}{h(z)} \) or~\( f(iz^{-1})\coloneqq \frac{h(-z)}{h(z)} \), we
    are getting~$|f(z)|\le1$ if~$\Im z\ge0$; thus, the above formula implies the convergence of
    the series~$ \sum_{\nu\in I,\ \nu\ge0}\left|\Re\frac{1}{\gamma_\nu}\right|$
    and~$\sum_{\nu\in I,\ \nu<0}\left|\Re\gamma_\nu\right|$, respectively.
\end{proof}
\begin{proof}[Proof of Theorem~\ref{th:main2}]
    Assume that~$H(p,q)$ is totally nonnegative; by Lemma~\ref{lemma:prop_S1} and
    Theorem~\ref{th:main1}, there is a function~$g(z)$ of the form~\eqref{eq:funct_gen_dtps}
    such that the ratios~$\frac{p(z)}{g(z)}$ and~$\frac{q(z)}{g(z)}$ represent functions with no
    common zeros. Moreover, according to Theorem~\ref{th:canonical_product} the estimate
    \[
        \max_{1 \le |z|\le R}\left(\left|\frac{p(z)}{g(z)}\right|
            +\left|\frac{p(z^{-1})}{g(z^{-1})}\right|\right)+
        \max_{1 \le |z|\le R}\left(\left|\frac{q(z)}{g(z)}\right|
            +\left|\frac{q(z^{-1})}{g(z^{-1})}\right|\right)
        <e^{\varepsilon R}
    \]
    holds true for an arbitrary~$\varepsilon>0$ and for~$R>1$ big enough; therefore,
    \begin{equation}\label{eq:f_estimate_0_inf}
        \max_{R^{-1}\le |z|\le R}\left|\frac{f(z)}{g(z^2)}\right|
        <e^{\varepsilon R^{2}}.
    \end{equation}
    Unless~$h(z)\coloneqq \frac{f(z)}{g(z^2)}$ is a constant, by
    Lemma~\ref{lemma:h_lhp_less_h_rhp} it satisfies~$|h(z)|>|h(-z)|$ in the right half of the
    complex plane. Since~$\frac{p(z)}{g(z)}$ and~$\frac{q(z)}{g(z)}$ do not vanish
    simultaneously, the function~$h(z)=\frac{q(z^2)}{g(z^2)}+z\frac{p(z^2)}{g(z^2)}$ has no
    purely imaginary zeros. Thus, the estimate~\eqref{eq:f_estimate_0_inf} implies the following
    representation:
    \begin{equation}\label{eq:h_is_f_over_g}
        h(z)
        =
        z^re^{Bz+\frac{B_0}z}
        \prod_{\lambda>0} \left(1+\frac{z}{\xi_\lambda}\right)
        \cdot
        \prod_{\nu>0} \left(1+\frac{z}{\gamma_\nu}\right)\left(1+\frac{z}{\overline \gamma_\nu}\right)
        \cdot
        \prod_{\lambda<0} \left(1+\frac{z^{-1}}{\xi_\lambda}\right)
        \cdot
        \prod_{\mu<0} \left(1+\frac{z^{-1}}{\gamma_\mu}\right)
                     \left(1+\frac{z^{-1}}{\overline \gamma_\mu}\right),
    \end{equation}
    where~$B,B_0\in\mathbb{R}$: the involved products are convergent since the
    sums~$\sum_{\nu>0}\left|\Re\gamma_\nu\right|$, \
    $\sum_{\nu<0}\left|\Re\gamma_\nu^{-1}\right|$, \ $\sum_{\lambda>0}\left|\Re \xi_\lambda\right|$
    and~$\sum_{\lambda<0}\left|\Re \xi_\lambda\right|$ are finite by Lemma~\ref{lemma:h_is_A}. In
    particular, for~$x>0$ we have
    \[
    \begin{aligned}
        h(x)&=
        x^re^{Bx+\frac{B_0}x}
        \prod_{\lambda>0} \left(1+\frac{x}{\xi_\lambda}\right)
        \cdot
        \prod_{\nu>0} \left(1+\frac{x}{\gamma_\nu}\right)
        \overline{\left(1+\frac{x}{\gamma_\nu}\right)}
        \cdot
        \prod_{\lambda<0} \left(1+\frac{x^{-1}}{\xi_\lambda}\right)
        \cdot
        \prod_{\mu<0} \left(1+\frac{x^{-1}}{\gamma_\mu}\right)
        \overline{\left(1+\frac{x^{-1}}{\gamma_\mu}\right)}\\
        &=
        x^r
        e^{Bx+\frac{B_0}x}
        \prod_{\lambda>0} \left(1+\frac{x}{\xi_\nu}\right)\cdot
        \prod_{\nu>0} \left|1+\frac{x}{\gamma_\nu}\right|^2\cdot
        \prod_{\lambda<0} \left(1+\frac{x^{-1}}{\xi_\lambda}\right)\cdot
        \prod_{\mu<0} \left|1+\frac{x^{-1}}{\gamma_\mu}\right|^2\!\!.
    \end{aligned}
    \]
    On the one hand, Theorem~\ref{th:canonical_product} for each~$\varepsilon>0$ implies
    \begin{equation*}
        |x|^r
        \prod_{\lambda>0} \left|1+\frac{x}{\xi_\lambda}\right|\cdot
        \prod_{\nu>0} \left|1+\frac{x}{\gamma_\nu}\right|^2\cdot
        \prod_{\lambda<0} \left|1+\frac{x^{-1}}{\xi_\lambda}\right|\cdot
        \prod_{\mu<0} \left|1+\frac{x^{-1}}{\gamma_\mu}\right|^2
        <e^{2\varepsilon\left|x+\frac1x\right|}
    \end{equation*}
    when~$\left|x+\frac 1x\right|$ is big enough; so, the
    ratio~$\frac{h(-x)}{h(x)}\sim e^{-2Bx-\frac{2B_0}x}$ grows to infinity as~$x\to+\infty$ or
    as~$x\to0+$ unless both conditions~$B\ge 0$ and~$B_0\ge0$ are satisfied. On the other hand,
    $\left|\frac{h(-x)}{h(x)}\right|<1$ for any~$x>0$ by Lemma~\ref{lemma:h_lhp_less_h_rhp}. As
    a result, the only consistent case is that~$f(z)$ can be represented as
    in~\eqref{eq:f_weierstrass}.

    Conversely, let $f(z)$ have the form~\eqref{eq:f_weierstrass}. Then there exists a
    function~$g(z)$ of the form~\eqref{eq:funct_gen_dtps} such that the
    ratio~$\frac{f(z)}{g(z^2)}$ satisfies~\eqref{eq:h_is_f_over_g} with~$B,B_0\ge0$. The
    polynomials
    \[
    \begin{aligned}
        h_n(z)
        =
        \left(1+\frac{Bz}n\right)^n\left(1+\frac{B_0z}{n}\right)^n
        &\cdot\!\!
        \prod_{0<\lambda\le n} \left(1+\frac{z}{\xi_\lambda}\right)
        \cdot\!\!
        \prod_{0<\nu\le n} \left(1+\frac{z}{\gamma_\nu}\right)\left(1+\frac{z}{\overline \gamma_\nu}\right)\\
        &\cdot\!\!\!\!
        \prod_{0>\lambda\ge -n} \left(z+\frac{1}{\xi_\lambda}\right)
        \cdot\!\!\!\!
        \prod_{0>\nu\ge-n} \left(z+\frac{1}{\gamma_{\nu}}\right)
                         \left(z+\frac{1}{\overline \gamma_{\nu}}\right)
    \end{aligned}
    \]
    are stable for each positive integer~$n$; if~$P_{n}(z)$ and~$Q_{n}(z)$ are defined
    by~$P_{n}(z^2)\coloneqq \frac 1{2z} (h_n(z) - h_n(-z))$
    and~$Q_{n}(z^2)\coloneqq \frac 12 (h_n(z) + h_n(-z))$, then the four-way infinite Hurwitz
    matrix~$H(P_{n},Q_{n})$ corresponding to~$h_n(z)$ is totally nonnegative by the Kemperman
    theorem~\cite{Kemperman,Dyachenko14}. Furthermore, the Hurwitz matrix generated by the
    rational functions
    \[
    \begin{aligned}
        z^r
        \left(1+\frac{Bz}n\right)^n\left(1+\frac{B_0z^{-1}}{n}\right)^n
        &\cdot\!\!
        \prod_{0<\lambda\le n} \left(1+\frac{z}{\xi_\lambda}\right)
        \cdot\!\!
        \prod_{0<\nu\le n} \left(1+\frac{z}{\gamma_\nu}\right)\left(1+\frac{z}{\overline \gamma_\nu}\right)\\
        &\cdot\!\!\!
        \prod_{0>\lambda\ge -n} \left(1+\frac{z^{-1}}{\xi_\lambda}\right)
        \cdot\!\!\!
        \prod_{0>\nu\ge-n} \left(1+\frac{z^{-1}}{\gamma_{\nu}}\right)
                         \left(1+\frac{z^{-1}}{\overline \gamma_{\nu}}\right)
    \end{aligned}
    \]
    coincides with~$H(P_{n},Q_{n})$ up to a shift in its indexation and, hence, it is also
    totally nonnegative. Since these rational functions converge to~$h(z)$ in every annulus
    centred at the origin as~$n\to\infty$, their Laurent coefficients converge to the
    coefficients of~$h(z)$. Thus, each minor of the matrix~$H(P,Q)$,
    where~$P(z^2)\coloneqq \frac 1{2z} (h(z) - h(-z))$
    and~$Q(z^2)\coloneqq \frac 12 (h(z) + h(-z))$, is nonnegative as a limit of nonnegative
    minors of~$H(P_{n},Q_{n})$. Recall that all minors of~$T(g)$ are nonnegative by
    Theorem~\ref{th:E-AESW}, so the Cauchy-Binet formula and the identity
    \[
    H(p,q)=H(P\cdot g,Q\cdot g)=H(P,Q)\cdot T(g)
    \]
    imply the total nonnegativity of~$H(p,q)$.
\end{proof}
\begin{proof}[Proof of Theorem~\ref{th:main3}]
    We reproduce the original proof of~\cite[Theorem~4]{HoltzKushelKhrushchev16} with minimal
    changes. Without loss of generality suppose that~$p_0(z)\not\equiv0$ and that~$H(p_m,p_0)$
    has a nonzero minor for some~$m>0$, which can be achieved by multiplying~$f(z)$ by some
    power of~$z$. By Theorem~\ref{th:main1},~$p_0(z)$ must be a non-trivial series convergent in
    some non-empty annulus and all other series~$p_1(z),\dots,p_{M-1}(z)$ converge in the same
    annulus (or trivial) since the corresponding matrices~$H(p_1,p_0),\dots,H(p_{M-1},p_0)$ are
    totally nonnegative. Let us keep the notation~$p_0(z),p_1(z),\dots,p_{M-1}(z)$ for functions
    represented by the same-name series. Then Theorem~\ref{th:main1} additionally implies that
    all poles and exponential factors of the functions~$p_0(z),p_1(z),\dots,p_{M-1}(z)$
    coincide. In other words, there exists a function~$g(z)$ of the
    form~\eqref{eq:funct_gen_dtps} such that the
    ratios~$\frac{p_0(z)}{g(z)},\dots,\frac{p_{M-1}(z)}{g(z)}$ can be represented as
    \[
    a_0{\prod_{\nu>0} \left(1+\dfrac{z}{\alpha_\nu}\right)}
    {\prod_{\nu<0}\left(1+\dfrac{z^{-1}}{\alpha_\nu}\right)}\quad\text{or}\quad
    C(z+\alpha_0)\prod_{\nu>0} \left(1+\dfrac{z}{\alpha_\nu}\right)
    \]
    with positive coefficients (see Lemma~\ref{lemma:prop_S1}). In particular, all these ratios
    behave at most subexponentially as~$z\to0+$ and as~$z\to+\infty$ by
    Theorem~\ref{th:canonical_product}; therefore, we can chose two entire functions~$q_1(z)$
    and~$q_2(z)$ of genus less than~$M$, so that~$f(z)$ has the required
    factorization~$g(z^M)\cdot q_1(z)\cdot q_2\left(\frac 1z\right)$.

    Let us prove that the sector~$C_M$ contains no zeros of~$f(z)$. On the one hand, for
    any~$z\in \big\{\zeta\in\mathbb C:0\le\arg \zeta<\frac{\pi}{M}\big\}$ and integers~$n,m$
    satisfying~$0\le n<m<M$, Theorem~\ref{th:main1} yields the
    inequality
    \[
        0\le \arg\frac{p_n(z^M)}{p_m(z^M)}< \pi,
        \quad\text{that is}\quad
        -\pi< \arg\frac{p_m(z^M)}{p_n(z^M)}\le 0.
    \]
    Since
    \[
    0 \le \arg z^{m-n} = (m-n) \arg z < \frac{m-n}{M}\pi<\pi,
    \]
    the argument (its principal value) of the product~$z^{m-n}\cdot\frac{p_m(z^M)}{p_n(z^M)}$ is
    equal to the sum of arguments of the factors, and thus
    \begin{equation}\label{eq:gen_hurw_ineq}
    \begin{aligned}
        (m-n)\arg z-\pi&<\arg\frac{z^{m-n}p_m(z^M)}{p_n(z^M)}\le(m-n)\arg z
        &&\quad\text{if}\quad m>n
        .
    \end{aligned}
    \end{equation}
    On the other hand, suppose that the condition
    \( 0=f(z)=p_0(z^M)+zp_1(z^M)+\dots+z^{M-1}p_{M-1}(z^M) \) holds true for some~$z$ varying in
    the chosen sector, which is is equivalent to
    \begin{equation}\label{eq:gen_hurw_eq}
        \sum_{n=1}^{M-1}\frac{z^np_n(z^M)}{p_0(z^M)}=-1
    \end{equation}
    due to~$p_0(z^M)\ne 0$. There are at least two nonzero summands on the left-hand side:
    otherwise this equality would contradict to~\eqref{eq:gen_hurw_ineq}. Let~$m$ and~$n$ be the
    indices of the nonzero summands with maximal and minimal arguments, respectively. Among the
    inequalities~\eqref{eq:gen_hurw_ineq}, we have
    \[
    m\arg z-\pi<
    \arg\frac{z^{m}p_m(z^M)}{p_0(z^M)}\le m\arg z
    \an
    n\arg z-\pi<
    \arg\frac{z^{n}p_n(z^M)}{p_0(z^M)}\le n\arg z,
    \]
    and therefore
    \[
    0=\max\big\{0,(m-n)\arg z-\pi\big\}
    \le
    \arg\frac{z^m p_m(z^M)}{p_0(z^M)}-\arg\frac{z^n p_n(z^M)}{p_0(z^M)}
    <
    (m-n)\arg z+\pi.
    \]
    Then the inequality
    \[
    \pi<\arg\frac{z^m p_m(z^M)}{p_0(z^M)}-\arg\frac{z^n p_n(z^M)}{p_0(z^M)}    
    \]
    implies~$m>n$ and
    \[
    \pi<
    \arg\frac{z^{m-n}p_m(z^M)}{p_n(z^M)}
    +2\pi<(m-n)\arg z+\pi,
    \quad\text{that is}
    \quad
    -\pi<
    \arg\frac{z^{m-n}p_m(z^M)}{p_n(z^M)}
    <(m-n)\arg z-\pi,
    \]
    which is inconsistent with~\eqref{eq:gen_hurw_ineq}; the reverse
    inequality
    \[
        \pi\ge\arg\frac{z^mp_m(z^M)}{p_0(z^M)}-\arg\frac{z^np_n(z^M)}{p_0(z^M)}
    \]
    implies that the
    cone~$\{0\}\cup\Big\{\zeta\in\mathbb{C}\setminus\{0\}:\arg\frac{z^mp_m(z^M)}{p_0(z^M)}\ge
    \arg\zeta\ge\arg\frac{z^np_n(z^M)}{p_0(z^M)}\Big\}$ is convex. This cone contains also all the
    ratios~$\frac{z^kp_k(z^M)}{p_0(z^M)}$, where~$k=1,\dots,M-1$, and hence
    \[
    \pi>\arg\frac{z^mp_m(z^M)}{p_0(z^M)}
    \ge\arg\sum_{k=1}^{M-1}\frac{z^kp_k(z^M)}{p_0(z^M)}
    \ge\arg\frac{z^np_n(z^M)}{p_0(z^M)}>-\pi,
    \]
    which contradicts to~\eqref{eq:gen_hurw_eq}. Consequently, there are no points~$z$ in the
    sector~$C_M\cap\{\Im z\ge0\}$ such that~$f(z)=0$. Complex conjugation gives that solutions
    to~$f(z)=0$ cannot belong to~$C_M\cap\{\Im z<0\}$ as well.
\end{proof}
\addsec{Acknowledgements}
The author is very grateful to Mikhail Tyaglov for his idea of Theorem~\ref{th:main3} and to
Olga Holtz for two questions which gave rise to the present publication.

{

}
\end{document}